\documentclass[11pt]{amsart}
\usepackage{amssymb,amscd,amsthm,verbatim}
\usepackage{amsbsy}
\usepackage{latexsym}

\date{\today}

\newcommand{\Z}{{\mathbb Z}}
\newcommand{\R}{{\mathbb R}}

\newcommand{\C}{{\mathbb C}}

\newcommand{\T}{{\mathbb T}}

\newcommand{\diam}{{\mathrm{diam}}}

\newtheorem{theorem}{Theorem}[section]
\newtheorem{lemma}[theorem]{Lemma}
\newtheorem{prop}[theorem]{Proposition}
\newtheorem{coro}[theorem]{Corollary}

\theoremstyle{definition}

\newtheorem*{definition}{Definition}

\newtheorem{remark}[theorem]{Remark}
\theoremstyle{plain}

\allowdisplaybreaks \numberwithin{equation}{section}

\DeclareMathOperator{\supp}{supp}

\begin{document}

\title[Schr\"odinger Operators with Ergodic+Random Potentials]{The Spectrum of Schr\"odinger Operators with Randomly Perturbed Ergodic Potentials}

\author[A.\ Avila]{Artur Avila}

\address{Institut f\"ur Mathematik, Universit\"at Z\"urich, Winterthurerstrasse 190, 8057 Z\"urich, Switzerland and IMPA, Estrada D. Castorina 110, Jardim Bot\^anico, 22460-320 Rio de Janeiro, Brazil}

\email{artur@math.sunysb.edu}

\author[D.\ Damanik]{David Damanik}

\address{Department of Mathematics, Rice University, Houston, TX~77005, USA}

\email{damanik@rice.edu}

\thanks{D.\ D.\ was supported in part by NSF grants DMS--1700131 and DMS--2054752}

\author[A.\ Gorodetski]{Anton Gorodetski}

\address{Department of Mathematics, University of California, Irvine, CA~92697, USA}

\email{asgor@uci.edu}

\thanks{A.\ G.\ was supported in part by NSF grant DMS--1855541.} 

\keywords{random Schr\"odinger operator, almost sure spectrum, rotation number}

\begin{abstract}
We consider Schr\"odinger operators in $\ell^2(\Z)$ whose potentials are given by the sum of an ergodic term and a random term of Anderson type. Under the assumption that the ergodic term is generated by a homeomorphism of a connected compact metric space and a continuous sampling function, we show that the almost sure spectrum arises in an explicitly described way from the unperturbed spectrum and the topological support of the single-site distribution. In particular, assuming that the latter is compact and contains at least two points, this explicit description of the almost sure spectrum shows that it will always be given by a finite union of non-degenerate compact intervals. The result can be viewed as a far reaching generalization of the well known formula for the spectrum of the classical Anderson model.
\end{abstract}

\maketitle

\section{Introduction}

In this paper we consider perturbations of ergodic Schr\"odinger operators in $\ell^2(\Z)$ by the addition of a random potential of Anderson type. We will for simplicity assume that both pieces of the potential are bounded.

It is known that with respect to the product measure, the spectrum is almost surely equal to the same set, which we will denote by $\Sigma_1$. A question of Bellissard asks whether it can be shown that $\Sigma_1$ has only finitely many gaps. Since, by general principles, $\Sigma_1$ cannot contain any isolated points, an equivalent formulation is the assertion that $\Sigma_1$ is given by a finite union of non-degenerate compact intervals.

It is in fact not obvious that in this generality, it is always true that $\Sigma_1$ even contains \emph{any} non-degenerate intervals, especially if the unperturbed ergodic model has a spectrum of Cantor type. This question, along with some preliminary results, was discussed and advertised in a recent paper by two of the authors \cite{DG22}.

The purpose of the present paper is to establish a full affirmative answer to Bellissard's question and prove the finiteness of the number of gaps of $\Sigma_1$ under the assumption that the hull of the ergodic piece is connected. This relatively weak assumption is satisfied by many popular models, including quasi-periodic potentials and potentials generated by skew-shifts and hyperbolic toral automorphisms. Beyond just the finiteness of the number of gaps of $\Sigma_1$, we even show how $\Sigma_1$ results in an explicit and simple way from the unpertubed almost sure spectrum $\Sigma_0$ and the topological support $S$ of the single-site measure $\nu$ generating the Anderson-type perturbation. This result in particular recovers the well known expression of the almost sure spectrum of the Anderson model, which in our setting corresponds to the case of a zero ergodic term.

\bigskip

Let us state our result precisely. The unperturbed model is given as follows. Given a compact metric space $X$, a homeomorphism $T : X \to X$, an ergodic Borel probability measure $\mu$ with full topological support, $\supp \mu = X$, and a sampling function $f \in C(X,\R)$, we generate potentials
$$
V_x(n) = f(T^n x), \; x \in X, \; n \in \Z
$$
and Schr\"odinger operators
$$
[H_x \psi](n) = \psi(n+1) + \psi(n-1) + V_x(n) \psi(n)
$$
in $\ell^2 (\Z)$. By the general theory of ergodic Schr\"odinger operators in $\ell^2(\Z)$, the spectrum of $H_x$, denoted by $\sigma(H_x)$, is almost surely independent of $x$. That is, there is a compact set $\Sigma_0$ such that
\begin{equation}\label{e.spectrumunpert}
\Sigma_0 = \sigma(H_x) \text{ for $\mu$-almost every } x \in X.
\end{equation}

The random perturbation is given by
$$
W_\omega(n) = \omega_n, \; \omega \in \Omega, \; n \in \Z,
$$
where $\Omega = (\supp \nu)^\Z$ and $\nu$ is a compactly supported probability measure on $\R$ with topological support
\begin{equation}\label{e.supportssm}
S := \supp \nu
\end{equation}
satisfying
$$
\# S \ge 2.
$$
Since the product of $\mu$ and $\tilde \mu := \nu^\Z$ is ergodic with respect to the product of $T$ and the left shift, there is, again by the general theory of ergodic Schr\"odinger operators in $\ell^2(\Z)$, a compact set $\Sigma_1$ such that
$$
\Sigma_1 = \sigma(H_x + W_\omega) \text{ for $\mu \times \tilde \mu$-almost every } (x,\omega) \in X \times \Omega.
$$
Since $\supp \mu \times \tilde \mu = X \times S^\Z$, we also have
\begin{equation}\label{e.spectrumpert}
\Sigma_1 = \bigcup_{(x,\omega) \in X \times S^\Z} \sigma(H_x + W_\omega),
\end{equation}
that is, the spectra corresponding to exceptional points can only be smaller than the almost sure spectrum. In particular,
\begin{equation}\label{e.specinclusion}
\sigma(H_x + W_\omega) \subseteq \Sigma_1 \text{ for every } (x,\omega) \in X \times S^\Z.
\end{equation}
For all the general results mentioned above and more background on ergodic Schr\"odinger operators, we refer the reader to \cite{D17, DF22b}.

Before stating our main theorem, we introduce the following operation on pairs of compact subsets of $\R$.

\begin{definition}
Suppose $A$ and $B$ are compact subsets of $\R$. We define the compact set $A \bigstar B$ as follows. If $\diam(A) \ge \diam(B)$, then $A \bigstar B = A + \mathrm{ch}(B)$, and if $\diam (A) < \diam (B)$, then $A \bigstar B = \mathrm{ch}(A) + B$. Here, $\diam(S)$ denotes the diameter and $\mathrm{ch}(S)$ denotes the convex hull of a compact $S \subset \R$, and $S_1 + S_2$ denotes the Minkowski sum $\{ s_1 + s_2 : s_1 \in S_1 , \, s_2 \in S_2 \}$.
\end{definition}

\begin{theorem}\label{t.main}
Consider the setting described above and assume that $X$ is connected. Then, we have
\begin{equation}\label{e.keyformula}
\Sigma_1 = \Sigma_0 \bigstar S.
\end{equation}
\end{theorem}

This theorem provides an affirmative answer to Bellissard's question:

\begin{coro}\label{c.main}
If $X$ is connected, then the almost sure spectrum $\Sigma_1$ is given by a finite union of non-degenerate compact intervals.
\end{coro}

\begin{proof}
This is an immediate consequence of \eqref{e.keyformula}.
\end{proof}

\begin{remark}
The standard Anderson model (see, e.g., \cite{K08, S11} for some introductory texts) arises in our setting if we set $V_x \equiv 0$ for all $x \in X$ (which can be accomplished by choosing the zero sampling function $f \in C(X,\R)$), which yields $\Sigma_0 = [-2,2]$. For the Anderson model, it is a classical result due to Kunz and Souillard \cite{KS80} (see also, e.g., \cite[Theorem~4.1]{D17}, \cite[Theorem~3.9]{K08}, \cite[Proposition~3.3]{S22}, and \cite[Theorem~2]{S11}) that
\begin{equation}\label{e.classicalformula}
\Sigma_1 = [-2, 2]+\supp \nu= \Sigma_0 + S.
\end{equation}
We remark here that \eqref{e.keyformula} recovers, and indeed vastly generalizes, \eqref{e.classicalformula}. To verify this, consider the two cases in question. If $\diam(S) \le 4 = \diam(\Sigma_0)$, then $\Sigma_1 = \Sigma_0 \bigstar S = [-2,2] + \mathrm{ch}(S) = \Sigma_0 + S$, where the last step follows from $\diam(S) \le 4$. On the other hand, if $\diam(S) > 4$, then $\Sigma_1 = \Sigma_0 \bigstar S = \mathrm{ch}([-2,2]) + S = \Sigma_0 + S$.
\end{remark}

\begin{remark}
In \cite{DG22} it was shown that for \emph{some} quasi-periodic potentials $\{V_x\}$, $\Sigma_1$ \emph{contains an} interval. Due to Corollary~\ref{c.main}, we now know that for \emph{all} quasi-periodic potentials $\{V_x\}$, $\Sigma_1$ is given by \emph{a finite union of non-degenerate compact intervals}. Beyond quasi-periodic potentials, which are generated by minimal translations on a finite-dimensional torus, our result covers other ergodic maps on finite-dimensional tori with fully supported ergodic measure. This includes, for example, (generalized) skew-shifts, for which it is also known that $\Sigma_0$ is generically nowhere dense \cite{ABD09, ABD12} and hence the topological structure of the almost sure spectrum changes markedly in all these cases when the random perturbation is turned on.
\end{remark}

\begin{remark}
Let us give an example showing that the formula \eqref{e.keyformula} may fail when $X$ is not assumed to be connected. Consider the case where the unperturbed potential is $2$-periodic, that is, we choose $X = \Z_2 = \Z/2\Z$, $T x = x+1$, $\mu = \frac12 \delta_0 + \frac12 \delta_1$. The perturbation is given by the Bernoulli-Anderson model, that is, $S$ has cardinality $2$. By the general theory of periodic Schr\"odinger operators it follows that for any $f:X\to \mathbb{R}$, the corresponding spectrum $\Sigma_0$ has diameter at least four and either is an interval or a union of two intervals. Therefore, if $\diam(S) \le 4$, $\Sigma_0 \bigstar S$ can have at most two connected components. However, explicit calculations (based on, e.g., \cite[Theorem~3.3.1]{DFG} or \cite[Lemma~2]{W}) show that for $f(0)=0$, $f(1)=7$, $S=\{-1, 1\}$, we have
$$
\Sigma_1=\left[\frac{5-\sqrt{65}}{2}, -1\right]\cup \left[\frac{7-\sqrt{41}}{2}, 1\right]\cup \left[6, \frac{7+\sqrt{41}}{2}\right]\cup \left[8, \frac{9+\sqrt{65}}{2}\right],
$$
which has four components and hence does not coincide with $\Sigma_0 \bigstar S$. We also mention that the topological structure of $\Sigma_1$ for operators of this kind was recently studied in \cite{DFG, W}, and it was shown that it is always a union of at most four intervals (for any values of $f(0), f(1)$, and any two-point set $S$).
\end{remark}

\begin{remark}
Some open questions about the topological structure of $\Sigma_1$ in the case of a disconnected $X$, which includes the case of the Anderson model with a periodic background, or about the topological structure of the essential spectrum in the case of the non-stationary Anderson model are formulated in \cite{DG22}, and we refer the reader to that paper for a more detailed discussion of them.
\end{remark}

\begin{remark}
While the case of an Anderson-type perturbation of a given ergodic reference operator is very natural, let us remark that Theorem~\ref{t.main} extends to a more general class of ergodic measures on $X \times S^\Z$. It is not necessary to consider the product $\mu \times \tilde \mu$. Moreover, it is not necessary to consider the product measure $\tilde \mu = \nu^\Z$ on $S^\Z$. The formula for $\Sigma_1$ holds whenever it is the almost sure spectrum associated with an ergodic measure on $X \times S^\Z$ that has full support. This can be seen either by inspection of the proof (which extends to measures of this kind) or by an application of the semi-continuity property of the spectrum with respect to strong convergence, which ensures that the almost sure spectrum of an ergodic family of Schr\"odinger operators is completely determined by the topological support of the push-forward measure on the space of realizations; compare \cite[Theorem~4.8.8]{DF22b}. For example, in the case when $T : X \to X$ is a uniquely ergodic homeomorphism of zero entropy (e.g., an irrational circle rotation), one can replace $\mu \times \tilde \mu$ by an ergodic measure on $X \times S^\Z$ with full support and zero entropy. The existence of such a measure follows, for example, from the following argument. Let $\mu_1$ be the unique invariant measure of the map $T:X\to X$. Let $\mu_2$ be an invariant ergodic measure with full support and zero entropy for the left shift $S^{\Z}\to S^{\Z}$; there is a residual set of such measures in the space of invariant measures of the left shift, see \cite{O63, Si70, Si71}.  Take $\mu_{1,2}$ to be an ergodic joining of the measures $\mu_1$ and $\mu_2$. Since both $\mu_1$ and $\mu_2$ have zero entropy, $\mu_{1,2}$ is also a measure of zero entropy. Finally, $\mu_{1,2}$ must have full support, $\text{\rm supp}\, \mu_{1,2}=X\times S^{\Z}$. Indeed, since $\text{\rm supp}\,\mu_2=S^{\Z}$, an $\omega$-limit set (with respect to the left shift) of any $\mu_2$-regular point contains every periodic point $p$ of the left shift. Together with the unique ergodicity of $T$, this implies that the $\omega$-limit set (with respect to the product of $T$ and the left shift) of any $\mu_{1,2}$-regular point must contain the whole leaf $X\times \{p\}$. Since such periodic leaves are dense in $X\times S^{\Z}$, we must have $\text{\rm supp}\, \mu_{1,2}=X\times S^{\Z}$. 
\end{remark}

The remainder of the paper is structured as follows. Theorem~\ref{t.main} will be proved in Section~\ref{sec.proof}, after having discussed, in Section~\ref{sec.prelim}, some ingredients used in the proof.

\section{Cocycles, Invariant Sections, and the Rotation Number}\label{sec.prelim}


%
%

In this section we discuss some objects that will play a crucial role in the proof of Theorem~\ref{t.main}. They are centered around, and related in spirit to, the Johnson-Schwartzman approach to gap labelling, but we will restrict our discussion to those aspects that are needed in the proof. For the results mentioned below, as well as more details and background, we refer the reader to \cite{DF22a, DF22b, J86, JONNF}.

Recall the framework that defines the unperturbed potentials: $X$ is a compact metric space, $T : X \to X$ is a homeomorphism, $\mu$ is an ergodic probability measure with full support, and $f : X \to \R$ is a continuous sampling function. We associate with these model data a one-parameter family of $\mathrm{SL}(2,\R)$ cocycles as follows: for $E \in \R$, we let
\begin{equation}\label{e.cocycle1}
A_E : X \to \mathrm{SL}(2,\R), \; A_E(x)=\begin{pmatrix}
           E-f(x) & -1 \\
           1 & 0 \\
         \end{pmatrix},
\end{equation}
and consider the cocycle
\begin{equation}\label{e.cocycle2}
(T,A_E) : X \times \mathbb{R}^2 \to X \times \mathbb{R}^2, (x,v) \mapsto (T x , A_E(x) v).
\end{equation}
Iterating the cocycle, we obtain maps $A_E^n : X \to \mathrm{SL}(2,\R)$ for $n \in \Z$ so that $(T,A_E)^n = (T^n, A_E^n)$. One says that $(T,A_E)$ is \emph{uniformly hyperbolic} if there are $C > 0$, $\lambda > 1$ such that $\inf_{x \in X} \|A_E^n(x)\| \ge C \lambda^{|n|}$. Since the ergodic measure $\mu$ has full support, the unperturbed almost sure spectrum can be characterized as follows:
\begin{equation}\label{e.johnsonstheorem}
\Sigma_0 = \{ E \in \R : (T,A_E) \text{ is not uniformly hyperbolic} \}.
\end{equation}

It is often convenient to consider the projectivization of $A_E(x)$ and regard it as a map from $\mathbb{RP}^1$ to $\mathbb{RP}^1$. Upon the natural identification of $\mathbb{RP}^1$ with $\T = \R/\Z$, we then arrive at a map from $\T$ to $\T$, which we denote by $g_E(x)$. An \emph{invariant section} of the cocycle $(T,A_E)$ is a continuous map $d : X \to \T$ such that for every $x \in X$, we have $g_E(x)(d(x)) = d(T x)$. If the cocycle $(T,A_E)$ is uniformly hyperbolic, it is well known that there are two invariant sections $d^s_E, d^u_E : X \to \T$, for which the associated vectors in $\R^2$ experience exponential decay in forward (resp., backward) time under iterations of the cocycle. These are called the \emph{stable} (rep., \emph{unstable}) section. We have
\begin{equation}\label{e.stableunstable}
d^s_E(x) \not= d^u_E(x) \text{ for every } x \in X,
\end{equation}
consistent with \eqref{e.johnsonstheorem}.

Since the cocycle $(T, A_E)$ is homotopic to the constant cocycle $(T,I)$ with the identity matrix $I$ (which is usually referred to as the cocycle being \emph{homotopic to the identity}), one can choose lifts $\tilde g_E(x) : \mathbb{R} \to \mathbb{R}$ (i.e., $\pi(\tilde g_E(x)(y)) = g_E(x)(\pi(y))$ with the canonical projection $\pi : \R \to \T$) that are continuous in both $x \in X$ and $E \in \mathbb{R}$. If $X$ is connected, any two such families of lifts must be the same up to an additive integer constant.

For $n \in \Z_+$, consider the composition 
$\tilde G_{x, E, n}= \tilde g_E(T^{n-1}(x)) \circ \ldots \circ \tilde g_E(x)$. Due to  \cite{ABD12, DF22a, DF22b,  J86, JONNF} and \cite[Theorem A.9]{GK21} we have the following:

\begin{prop}\label{p.rotationnumberdefandprops}
For any $E\in \R$, the rotation number
$$
\rho(E)=\lim_{n\to \infty}\int_X\frac{\tilde G_{x, E, n}(y)}{n}d\mu(y)
$$
exists. Moreover, for $\mu$-a.e. $x\in X$ and any $y\in \mathbb{R}$, we have
$$
\frac{\tilde G_{x, E, n}(y)}{n}\to \rho(E)\ \ \text{as}\ \ \ n\to \infty.
$$
The function $E \mapsto \rho(E)$ is continuous and monotone. Moreover, it is constant on an interval $(E_1, E_2)$ if and only if $(T,A_E)$ is uniformly hyperbolic for every $E\in (E_1, E_2)$.
\end{prop}

\begin{remark}
Notice that since the lifts could be shifted by the same integer constant, in the case of a connected $X$, the rotation number is only defined up to an integer.
\end{remark}



In the proof of Theorem~\ref{t.main} we will crucially rely on the following statement, whose proof will be given at the end of this section:

\begin{prop}\label{p.keyprop}
Assume that $X$ is connected. Suppose $E_1, E_2 \in \R$ are such that $E_1 < E_2$ and the cocycles $(T, A_{E_1})$ and $(T, A_{E_2})$ are uniformly hyperbolic. Then either $[E_1, E_2]\cap \Sigma_0=\emptyset$, or $\Sigma_0\subset (E_1, E_2)$, or the unstable sections $d^{u}_{E_1}$ and $d^{u}_{E_2}$ are not homotopic.
\end{prop}

The next proposition does not assume the connectedness of $X$:

\begin{prop}\label{p.homotopic}
Suppose $d^1 : X \to \mathbb{T}$ and $d^2 : X \to \mathbb{T}$ are continuous. If we have $d^1(x) \ne d^2(x)$ for all $x\in X$, then $d^1$ and $d^2$ are homotopic.
\end{prop}
\begin{proof}
Continuously turn $d^1(x)$ counterclockwise until it ``hits'' $d^2(x)$.
\end{proof}

As a corollary (of Proposition~\ref{p.homotopic} and \eqref{e.stableunstable}) we get the following simple fact that we state explicitly:

\begin{prop}\label{p.stableunstablehomotopic}
If $E \not \in \Sigma_0$, then the unstable section $d^{u}_{E}$ and the stable section $d^{s}_{E}$ of $(T,A_E)$ are homotopic.
\end{prop}

The following statement is just a convenient reformulation of the definition of the rotation number of a cocycle in the case where an invariant section exists (not necessarily a stable/unstable section of a hyperbolic cocycle, although that is the context we are most interested in):

\begin{prop}\label{p.1}
Suppose $X$ is connected and the cocycle $(T, A)$ has an invariant section $d : X \to \mathbb{T}$. Then the displacement function
$$
\varphi_{d, A}(x)=\tilde g{(x)}(y)-y,
$$
where $y\in \pi^{-1}(d(x))$, does not depend on the choice of $y\in  \pi^{-1}(d(x))$, and the rotation number $\rho(T, A)$ is given by
$$
\rho(T, A)=\int_X\varphi_{d, A}(x) \, d\mu(x) \!\!\! \mod 1.
$$
\end{prop}
\begin{remark}
Notice that when $X$ is connected, the choice of a different family of lifts $\tilde g{(x)}$ will only change the integral $\int_X\varphi_{d, A}(x) \, d\mu(x)$ by an integer constant.
\end{remark}

\begin{prop}\label{p.3}
In the context of Proposition \ref{p.1}, suppose that another cocycle $(T, A')$ is such that it has the same $d : X \to \mathbb{T}$ as an invariant section. Then
$$
\rho(T, A)=\rho(T, A') \!\!\! \mod 1.
$$
\end{prop}

\begin{proof}
Notice that $\varphi_{d, A}(x) - \varphi_{d, A'}(x)$ is a continuous integer-valued function. Since $X$ is connected, the result follows.
\end{proof}

\begin{prop}\label{p.2}
Suppose that $X$ is connected, $(T, A)$ is an $\mathrm{SL}(2, \mathbb{R})$ cocycle that is homotopic to the identity,  $\alpha : X \to \mathbb{R}$ is continuous, and the cocycle $(T, C)$ is given by $C(x) = R_{-\alpha(T(x))}A(x)R_{\alpha(x)}$,
where $R_\theta \in \mathrm{SL}(2, \mathbb{R})$ is the rotation in $\R^2$ by angle $\theta$.
Then $\rho(T, C) = \rho(T, A) \!\! \mod 1$.
\end{prop}

\begin{proof}
Denote by $\{g_A(x)\}$ the family of projectivizations of $\{A(x)\}$, and let
$\{\tilde g_{A}(x) \}$ be a continuous family of lifts of the maps $\{g_A(x)\}$; such a family of lifts exists since the cocycle $(T, A)$ is homotopic to the identity. Similarly to the notation introduced above for the Schr\"odinger cocycles, set $\tilde G_{x, A, n}(y)=\tilde g_A{(T^{n-1}(x))}\circ \ldots \circ \tilde g_A{(x)}$.  Then
$$
\tilde g_{C}(x)(y)=\tilde g_{A}(x)(y+\alpha(x)) - \alpha(T(x))
$$
form a family of lifts of projectivizations of $\{C(x)\}$, and
\begin{align*}
\tilde G_{x, C, n}(y) & = \tilde g_{T^{n-1}(x), C}\circ \ldots \circ \tilde g_{x, C}(y) \\
& =\tilde g_{T^{n-1}(x), A}\circ \ldots \circ \tilde g_{x, A}(y+\alpha(x))-\alpha(T^n(x)).
\end{align*}
Hence, for $\mu$-regular $x\in X$, we have
$$
\rho(T, C)=\lim_{n\to \infty}\frac{\tilde G_{x, C, n}(y)}{n}=\lim_{n\to \infty}\frac{\tilde G_{x, A, n}(y)}{n}=\rho(T, A),
$$
concluding the proof.
\end{proof}

Together, Propositions \ref{p.2} and \ref{p.3} imply the following:

\begin{prop}\label{p.homotopicsamegap}
Assume that $X$ is connected. If the unstable sections $d^{u}_{E_1}$ and $d^{u}_{E_2}$ are homotopic, then $\rho(E_1)=\rho(E_2)\!\! \mod 1$.
\end{prop}

\begin{proof}
If $d^{u}_{E_1}$ and $d^{u}_{E_2}$ are homotopic, the cocycle $(T, A_{E_1})$ is conjugate to a cocycle for which the section $d^{u}_{E_2}$ is invariant (and which, due to Proposition~\ref{p.3}, has the same rotation number). On the other hand, due to Proposition~\ref{p.2} it also must have the same rotation number as $(T, A_{E_2})$.
\end{proof}

\begin{proof}[Proof of Proposition~\ref{p.keyprop}]
Given $E_1, E_2 \in \R$ with $E_1 < E_2$ so that the cocycles $(T, A_{E_1})$ and $(T, A_{E_2})$ are uniformly hyperbolic, it follows from \eqref{e.johnsonstheorem} that $E_1$ and $E_2$ belong to the complement of $\Sigma_0$. If they belong to the same connected component of $\Sigma_0^c \cap \R$ (i.e., the same gap), then we have $[E_1, E_2] \cap \Sigma_0=\emptyset$. If $E_1 < \min \Sigma_0$ and $E_2 > \max \Sigma_0$, then $\Sigma_0\subset (E_1, E_2)$. Otherwise, we must have that $E_1$ and $E_2$ belong to different gaps of $\Sigma_0$, one of which is bounded. In this case, it follows from Proposition~\ref{p.rotationnumberdefandprops} and Proposition~\ref{p.homotopicsamegap} that the unstable sections $d^{u}_{E_1}$ and $d^{u}_{E_2}$ are not homotopic.
\end{proof}

\section{Proof of the Main Result}\label{sec.proof}

In this section we prove Theorem~\ref{t.main}. We will denote the complement of a subset $\mathcal{S}$ of $\R$ by $\mathcal{S}^c$, that is, $\mathcal{S}^c = \R \setminus \mathcal{S}$. We emphasize this because sometimes spectra are naturally considered as subsets of $\C$, but for the discussion below it is not necessary to move off the real axis.

Recall that $\Sigma_0$ denotes the unperturbed spectrum, compare \eqref{e.spectrumunpert}, and $\Sigma_1$ denotes the almost sure spectrum after adding the random perturbation, compare \eqref{e.spectrumpert}. Recall also that $S$ denotes the topological support of the single-site measure generating the random perturbation; see \eqref{e.supportssm}.

Before we start the formal proof, we would like to remind the reader that the addition of a constant to the potential is equivalent to a shift in the energy. Therefore, if one takes an energy in $\Sigma_0$, the addition of a constant that belongs to $S$ must give us an energy in  $\Sigma_1$. It is now key to our argument to observe the following. Start with the Schr\"odinger cocycle $(T, A_E)$ associated with the unperturbed ergodic potential, compare \eqref{e.cocycle1}--\eqref{e.cocycle2}, and suppose that there are two constants from the support $S$ of the random perturbation such that adding any one of them to an energy outside of $\Sigma_0$ produces a uniformly hyperbolic cocycle and also such that the unstable sections of these two cocycles are not homotopic. Then there is a point in the phase space where the stable direction of one cocycle coincides with the unstable direction of the other. Consider the potential generated by the $T$-orbit of that point. Then adding one constant to all values of the potential on the right half line and another constant to the potential on the left half line gives a sequence of matrices that are hyperbolic on each of the half lines, but such that the most contracting vector of the products of matrices to the right coincides with the most contracting vector to the left, hence its images form an eigenfunction, and therefore the energy in question must be in $\Sigma_1$.

We present the formal argument in the proof of the following statement:

\begin{lemma}\label{l.main1}
If $E \in \Sigma_1^c$, then $E - S \subseteq \Sigma_0^c$ and all $E' \in E - S$ have homotopic unstable sections with respect to the unperturbed cocycle at energy $E'$.
\end{lemma}

\begin{proof}
We show the contrapositive. That is, if $E - S \not\subseteq \Sigma_0^c$ or if $E - S \subseteq \Sigma_0^c$ and there are $v, v' \in S$ that have non-homotopic unstable sections at energies $E - v$ and $E - v'$ with respect to the respective unperturbed cocycles, then $E \in \Sigma_1$.

Consider first the case $E - S \not\subseteq \Sigma_0^c$. Then, there is $v \in S$ such that $E - v \in \Sigma_0$. This shows that the constant realization $W_\omega \equiv v$ is such that $E \in \sigma(H + V_x + W_\omega)$ for every $x \in X_0$ with $\sigma(H_x) = \Sigma_0$. Since this set of $x$'s has full $\mu$ measure, and translates of $\tilde \mu$-almost every $\tilde \omega$ can approximate $\omega$, a strong approximation argument then implies that $E \in \Sigma_1$, as desired.

In the other case, $E - S \subseteq \Sigma_0^c$ and there are $v, v' \in S$ such that $E- v$ and $E - v'$ have non-homotopic unstable sections with respect to the respective unperturbed cocycles. Consider the random realization
$$
W_\omega(n) = \begin{cases} v & n \in \Z_-, \\ v' & n \in \Z_+. \end{cases}
$$
Since the stable and unstable sections of the unperturbed cocycle for fixed energy are homotopic by Proposition~\ref{p.stableunstablehomotopic}, by assumption we have that the unstable section for energy $E - v$ and the stable section for energy $E - v'$ are non-homotopic (they exist due to $E - S \subseteq \Sigma_0^c$). By Proposition~\ref{p.homotopic} there exists $x \in X$ such that $d^u_{E-v}(x) = d^s_{E-v'}(x)$. This shows that the Schr\"odinger operator with potential $V_x + W_\omega$ for these particular choices for $x$ and $\omega$ possesses an exponentially localized eigenvector at energy $E$. Thus, we have $E \in \sigma(H_x + W_\omega)$, and hence by \eqref{e.specinclusion}, we have $E \in \Sigma_1$.
\end{proof}

\begin{lemma}\label{l.main}
Assume that $X$ is connected. Then, $E \in \Sigma_1^c$ if and only if either $E - S$ is contained in a gap of $\Sigma_0$ or $\Sigma_0$ is contained in a gap of $E - S$.
\end{lemma}

\begin{proof}
For the first direction we suppose that $E \in \R$ is such that neither $E - S$ is contained in a gap of $\Sigma_0$, nor $\Sigma_0$ is contained in a gap of $E - S$. We need to show that $E \in \Sigma_1$.

One possibility is that $E - S$ intersects $\Sigma_0$. By the argument given above in the proof of Lemma~\ref{l.main1}, it follows that $E \in \Sigma_1$, as desired.

The other possibility is that neither set is contained in a gap of the other, but they still have empty intersection. In this case one can find $v,v' \in S$ such that $E - v$ and $E - v'$ belong to different gaps of $\Sigma_0$, one of which must be an interior (i.e., bounded) gap. Proposition~\ref{p.keyprop} now shows that the unstable sections at these two energies are non-homotopic. Thus, by Lemma~\ref{l.main1} we find $E \in \Sigma_1$, again as desired.

\medskip

For the reverse direction we suppose that $E \in \R$ is such that either $E - S$ is contained in a gap of $\Sigma_0$, or $\Sigma_0$ is contained in a gap of $E - S$. We need to show that $E \in \Sigma_1^c$. 
It is clear that $\Sigma_1 \subseteq \mathrm{ch}(S)+\Sigma_0$. Thus, if $E - S$ is contained in a gap of $\Sigma_0$, then $E - \mathrm{ch}(S)$ is also contained in a gap of $\Sigma_0$, and hence $E$ cannot be in $\mathrm{ch}(S)+\Sigma_0$, so $E\not\in \Sigma_1$.

To see the other implication, we can assume without loss of generality that $r := \max \Sigma_0 = - \min \Sigma_0$ (otherwise shift appropriately and subsume the necessary translate in $S$). By self-adjointness, we therefore must have
\begin{equation}\label{e.Hx}
\|H_x\| = r \ \text{ for}\ \ \mu-\text{almost every}\ \ x \in X.
\end{equation}
Arguing in a similar way as before, the addition of $H_x$ to $W_\omega$ can shift the edge of a spectral gap by no more than $r$ (for $\mu$-almost every $x \in X_0$). Since $S$ is the spectrum of the multiplication operator $W_\omega$ for $\tilde \mu$-almost every $\omega \in \Omega$, it follows that $E \in \Sigma_1^c$ if $\Sigma_0$ is contained in a gap of $E - S$. Indeed, if $\Sigma_0$ is contained in a gap of $E-S$, then $\mathrm{ch}(\Sigma_0)$ is contained in the same gap of $E - S$ as well, which implies that
$$
E \not \in S + \mathrm{ch}(\Sigma_0) = S + [-r, r].
$$
Due to (\ref{e.Hx}), this implies that $E \in \Sigma_1^c$.
\end{proof}


\begin{proof}[Proof of Theorem~\ref{t.main}]
Consider first the case where $\diam(S) \le \diam(\Sigma_0)$. Then, by Lemma~\ref{l.main}, $E \in \Sigma_1^c$ if and only if $E - S$ is contained in a gap of $\Sigma_0$ (as the other case is impossible). But this in turn is equivalent to the statement that $E - \mathrm{ch}(S)$ is contained in a gap of $\Sigma_0$. It follows that an $E \in \R$ obeys $E \notin \Sigma_1$ if and only if $E \notin \Sigma_0 + \mathrm{ch}(S)$, whence $\Sigma_1 = \Sigma_0 \bigstar S$ in this case.

In the case where $\diam(S) > \diam(\Sigma_0)$, we argue similarly. By Lemma~\ref{l.main}, $E \in \Sigma_1^c$ if and only if $\Sigma_0$ is contained in a gap of $E - S$. This in turn is equivalent to the statement that $\mathrm{ch}(\Sigma_0)$ is contained in a gap of $E - S$. It follows that an $E \in \R$ obeys $E \notin \Sigma_1$ if and only if $E \notin \mathrm{ch}(\Sigma_0) + S$, whence $\Sigma_1 = \Sigma_0 \bigstar S$ in this case as well.
\end{proof}


\end{document}